\documentclass[11pt]{amsart}
\usepackage{amsmath,amssymb,lscape,amsthm}
\usepackage[mathscr]{eucal}

\newtheorem{dfn}{Definition}
\newtheorem{thm}{Theorem}
\newtheorem{prop}{Proposition}
\newtheorem{lem}{Lemma}
\newtheorem{cor}{Corollaly}
\newtheorem{rem}{Remark}

\thanks{The second author was partly supported by the Grant-in-Aid for Science Research (C) No. 26400073, JSPS}
\begin{document}
\title{Area-minimizing cones over minimal embeddings of R-spaces}
\author{Shinji Ohno}
\address{Department of Mathematics and Information Sciences,
Tokyo Metropolitan University,
Minami-Osawa, Hachioji-shi,
Tokyo, 192-0397 Japan}
\email{oono-shinji@ed.tmu.ac.jp}

\author{Takashi Sakai}
\email{sakai-t@tmu.ac.jp}

\maketitle

\begin{abstract}
In this paper, by constructing area-nonincreasing retractions,
we prove area-minimizing properties of some cones
over minimal embeddings of R-spaces.
\end{abstract}

\section{Introduction}
Let $C_{B}$ be the cone over a submanifold $B$ of the unit sphere $S^{n-1}$ in $\mathbb{R}^{n}$.
The cone $C_B$ is minimal in $\mathbb{R}^n$ if and only if $B$ is minimal in $S^{n-1}$.
We call a cone $C_B$ area-minimizing if its truncated cone $C_B^1$
has the least area among all integral currents with the same boundary $B$. 
Solutions of Plateau's problem can have singularities as integral currents.
At an isolated conical singularity, the tangent cone is area-minimizing.
Hence, in order to understand such singularities, we should study area-minimizing properties of minimal cones.

Lawlor \cite{L} gave a sufficient condition, so-called the {\it curvature criterion},
for a cone to be area-minimizing, using an area-nonincreasing retraction.
With this criterion, he obtained a complete classification of area-minimizing cones
over products of spheres and the first examples of area-minimizing cones over
nonorientable manifolds.
Kerckhove \cite{Ke} proved that some cones over isolated orbits of the adjoint representations of
$\mathrm{SU}(n)$ and $\mathrm{SO}(n)$ are area-minimizing.
A symmetric R-space can be minimally embedded in the sphere in a canonical way.
Hirohashi, Kanno and Tasaki \cite{HKT} constructed area-nonincreasing retractions
onto the cones over symmetric R-spaces associated with symmetric pairs of type $\mathrm{B}_l$.
Furthermore, Kanno \cite{Ka} proved that cones over some symmetric R-spaces are area-minimizing.

In this paper, we study area-minimizing properties of
cones over minimal embeddings of R-spaces, not only symmetric R-spaces.
In Theorem \ref{thm:retr}, we give a construction of retractions generalizing the method given in \cite{HKT}.
Applying this theorem we give some examples of area-minimizing cones over minimal embeddings of R-spaces.
In Section \ref{reducible}, we discuss area-minimizing properties of cones
over products of R-spaces.

\section{Preliminaries}\label{Preliminaries}
\subsection{Area-minimizing cones.} 
	Let $B$ be a submanifold of the unit sphere $S^{n-1}$ in $\mathbb{R}^{n}$.
	We define the cone $C_{B}$ and the truncated cone $C_{B}^{1}$ over $B$ by
	\begin{eqnarray*}
	C_{B}&=& \{ tx \in \mathbb{R}^{n} \mid 0\leq t,x\in B \},\\
	C_{B}^{1}&=& \{ tx \in \mathbb{R}^{n} \mid 0\leq t \leq 1,x\in B\}.
	\end{eqnarray*}
	Both $C_{B}$ and $C_{B}^{1}$ have an isolated singularity at the origin $0 \in \mathbb{R}^{n}$.
	\begin{dfn}\rm
	A cone $C_{B}$ is called area-minimizing if $C_{B}^{1}$ has the least area among all integral currents with boundary B.
	\end{dfn}
	Let $V$ and $W$ be two vector spaces with inner products, and let $F:V\to W$ be a linear map.
	Suppose $\dim V=n\geq \dim W =m$.
	We define the Jacobian $JF$ of $F$ by
	$$JF=\sup \{ \|F(v_{1}) \wedge \cdots \wedge F(v_{m}) \| \},$$
	where $\{ v_{1},\ldots ,v_{m}\}$ runs over all orthonormal systems of $V$.
	If $F$ is not surjective, then $JF=0$. 
	If $F$ is surjective, then   
	$$JF=\|F(v_{1}) \wedge \cdots \wedge F(v_{m}) \|$$
	for any orthonormal basis of $(\ker F)^{\perp}$.

	\begin{dfn}\rm
	A retraction $\Phi: \mathbb{R}^{n} \to C_{B}$ is called differentiable if $\Phi: \mathbb{R}^{n}\setminus \Phi^{-1}(0) \to C_{B}\setminus \{0 \} $ is $C^1$.
	A differentiable retraction $\Phi$ is called area-nonincreasing if 
	$J(d\Phi)_{x} \leq 1$ holds for all $x \in \mathbb{R}^{n}\setminus \Phi^{-1}(0)$.
	\end{dfn}

\begin{prop}
	Let $B$ be a compact submanifold of the unit sphere $S^{n-1}$ in $\mathbb{R}^{n}$. 
	Suppose that there exists an area-nonincreasing retraction $\Phi$ from $\mathbb{R}^{n}$ to $C_{B}$.
	Then $C_{B}$ is area-minimizing.
	\end{prop}
	\begin{proof}
	Let $S$ be an integral current which has the same boundary $B$ as $C_{B}^{1}$.
	Since $\Phi(S) \supset C_{B}^{1}$, we have 
	\begin{eqnarray*}
	\mathrm{Vol}(C_{B}^{1}) &\leq &\mathrm{Vol}(\Phi (S)) =\mathrm{Vol}(\Phi (S)\setminus \{ 0\}) \leq  \int_{S \setminus \Phi^{-1}(0)} \| d\Phi (e_{1}\wedge \cdots \wedge e_{k})\| d\mu_{S}\\
	&\leq & \int_{S \setminus \Phi^{-1}(0)} J(d\Phi)_{x} d\mu_{S} \leq  \int_{S \setminus \Phi^{-1}(0)} 1 d\mu_{S} \leq  \int_{S} 1 d\mu_{S} = \mathrm{Vol} (S),\\
	\end{eqnarray*}
	where $\{ e_{1}, \ldots , e_{k} \}$ is an orthonormal frame of $S$.
	\end{proof}
	If $C_{B}$ is area-minimizing, then $C_{B}$ is minimal in $\mathbb{R}^{n}$.
	Therefore, to find area-minimizing cones, it suffices to consider cones over minimal submanifolds of $S^{n-1}$.
	For this purpose, we use $s$-representations, which are the linear isotropy representations of Riemannian symmetric spaces.
	
\subsection{Riemannian symmetric pairs and restricted root systems.} 
	Let $G$ be a connected Lie group and $\theta$ be an involutive automorphism of $G$. 
	We denote by $F(\theta , G)$ the fixed point set of $\theta$,
	and we denote by $F(\theta ,G)_{0}$ the identity component of $F(\theta , G)$.
	For a closed subgroup $K$ of $G$, the pair $(G,K)$ is said to be a Riemannian symmetric pair if $F(\theta , G)_{0} \subset K \subset F(\theta , G)$
	and $\mathrm{Ad}(K)$ is compact.
	Let $(G,K)$ be a Riemannian symmetric pair, and $\mathfrak{g}$ and $\mathfrak{k}$ be Lie algebras of $G$ and $K$, respectively. 
	We immediately see that  
	$$\mathfrak{k}= \{X \in \mathfrak{g} \mid d\theta (X)=X \}.$$
	We put
	$$\mathfrak{m}= \{X \in \mathfrak{g} \mid d\theta (X)=-X \}.$$
	We denote by $\langle \cdot , \cdot \rangle$ an inner product on $\mathfrak{g}$ which is invariant under the actions of $\mathrm{Ad}(K)$ and $d\theta$.
	Then $\langle \cdot , \cdot \rangle$ induces a left-invariant metric on $G$ and a $G$-invariant metric on $M=G/K$ to be a Riemannian symmetric space,
	which are denoted by the same symbol $\langle \cdot , \cdot \rangle$.
	Since $d\theta$ is involutive, we have an orthogonal direct sum decomposition of $\mathfrak{g}$:
	$$\mathfrak{g}=\mathfrak{k}+\mathfrak{m}.$$
	This decomposition is called the canonical decomposition of $(\mathfrak{g}, \mathfrak{k})$.\
	For the origin $o \in G/K$, we can identify the tangent space $T_{o}(G/K)$ with $\mathfrak{m}$ by the differential of the natural projection $\pi :G \to G/K$.
	
	In this paper, we consider only Riemannian symmetric spaces of compact type. 
	We suppose that $G$ is compact and semisimple.
	Take and fix a maximal abelian subspace $\mathfrak{a}$ in $\mathfrak{m}$ and a maximal abelian subalgebra $\mathfrak{t}$ in $\mathfrak{g}$ including $\mathfrak{a}$. 
	For $\lambda \in \mathfrak{t}$, we put 
	$$\tilde{\mathfrak{g}}_{\lambda} =\{X \in \mathfrak{g}^{\mathbb{C}} \mid [H,X] = \sqrt{-1} \langle \lambda ,H \rangle X \ (H \in \mathfrak{t})\}$$
	and define the root system $\tilde{R}$ of $\mathfrak{g}$ by
	$$\tilde{R}= \{\lambda \in \mathfrak{t} \setminus \{0\} \mid \tilde{\mathfrak{g}}_{\lambda} \neq \{0\} \}.$$
	For $\lambda \in \mathfrak{a}$, we put 
	$$\mathfrak{g}_{\lambda} =\{X \in \mathfrak{g}^{\mathbb{C}} \mid [H,X] = \sqrt{-1} \langle \lambda ,H \rangle X \ (H \in \mathfrak{a})\}$$
	and define the restricted root system $R$ of $(\mathfrak{g},\mathfrak{k})$ by 
	$$R= \{\lambda \in \mathfrak{a} \setminus \{0\} \mid \mathfrak{g}_{\lambda} \neq \{0\} \}.$$
	Denote the orthogonal projection from $\mathfrak{t}$ to $\mathfrak{a}$ by $H \mapsto \overline{H}$.
	We extend a basis of $\mathfrak{a}$ to that of $\mathfrak{t}$ and define a lexicographic orderings $>$ on $\mathfrak{a}$ and $\mathfrak{t}$ with respect to these basis.
	Then for $H \in \mathfrak{t}$, $\overline{H} >0$ implies $H>0$.
	We denote by $\tilde{F}$ the fundamental system of $\tilde{R}$ with respect to $>$,
	by $F$ the fundamental system of $R$ with respect to $>$.
	We define 
	$$\tilde{R}_{0}=\{\lambda \in \tilde{R}\mid \overline{\lambda} =0\},\quad \tilde{F}_{0}=\{\alpha \in \tilde{F}\mid \overline{\alpha} =0\}.$$
	Then we have 
	$$
	R=\{\overline{\lambda } \mid \lambda \in \tilde{R} \setminus \tilde{R}_{0}\},\quad F=\{\overline{\alpha } \mid \alpha \in \tilde{F} \setminus \tilde{F}_{0}\}.
	$$
	We denote the set of positive roots  by 
	$$
	\tilde{R}_{+}=\{\lambda \in \tilde{R} \mid \lambda >0\},\quad R_{+}=\{\lambda \in R \mid \lambda >0\}
	.$$
	We put
	$$\mathfrak{k}_{0}=\{X \in \mathfrak{k} \mid [H, X]=0 \ (H \in \mathfrak{a} )\}$$
	and for each $\lambda \in R_+$
	$$
	\mathfrak{k}_{\lambda}=\mathfrak{k}\cap (\mathfrak{g}_{\lambda}+\mathfrak{g}_{-\lambda}),\quad \mathfrak{m}_{\lambda}=\mathfrak{m}\cap (\mathfrak{g}_{\lambda}+\mathfrak{g}_{-\lambda}).
	$$
	We then have the following lemma.
	\begin{lem}[\cite{T}]\label{lem:decomp}
	(1)\ We have orthogonal direct sum decompositions:
	$$\displaystyle \mathfrak{k}=\mathfrak{k}_{0}+ \sum_{\lambda \in R_{+}} \mathfrak{k}_{\lambda}\ ,\quad \mathfrak{m}=\mathfrak{a}+ \sum_{\lambda \in R_{+}} \mathfrak{m}_{\lambda}\ .$$
	(2)\ For each $\mu \in \tilde{R}_{+} \setminus \tilde{R}_{0}$ there exist $S_{\mu} \in \mathfrak{k}$ and $T_{\mu} \in \mathfrak{m}$
	such that 
	$$\{S_{\mu} \mid \ \mu \in \tilde{R}_{+} ,\ \overline{\mu}=\lambda \},\quad \{T_{\mu}\mid \mu \in \tilde{R}_{+} ,\ \overline{\mu}=\lambda \}$$
	are, respectively, orthonormal bases of $\mathfrak{k}_{\lambda}$ and $\mathfrak{m}_{\lambda}$ 
	and that for any $H \in \mathfrak{a}$
	$$[H, S_{\mu}]= \langle \mu , H\rangle T_{\mu},\quad [H, T_{\mu}]=- \langle \mu , H\rangle S_{\mu}.$$
	\end{lem}
	For each $\lambda \in R_{+}$ we put $m(\lambda ) = \dim \mathfrak{m}_{\lambda}=\dim \mathfrak{k}_{\lambda} $. 
	$m(\lambda)$ is called the multiplicity of $\lambda$.
	We define a subset $D$ of $\mathfrak{a}$ by 
	$$D= \bigcup_{\lambda \in R_{+}} \{ H \in \mathfrak{a}\ | \ \langle \lambda , H\rangle =0 \}.$$
	Each connected component of $\mathfrak{a} \setminus D$ is called a Weyl chamber.
	We define the fundamental Weyl chamber $\mathcal{C}$ by 
	$$\mathcal{C}= \{H \in \mathfrak{a} \ | \ \langle \alpha , H \rangle >0 \ (\alpha \in F)\}.$$
	The closure of $\mathcal{C}$ is given by 
	$$\overline{\mathcal{C}}= \{ H \in \mathfrak{a} \ | \ \langle \alpha , H \rangle \geq 0 \ (\alpha \in F)\}.$$
	For each subset $\Delta \subset F$, we define a subset $\mathcal{C}^{\Delta } \subset \overline{\mathcal{C}}$ by 
	$$\mathcal{C}^{\Delta }= \{H \in \overline{\mathcal{C}} \ | \ \langle \alpha , H \rangle >0 \ (\alpha \in \Delta) ,\ \langle \beta , H \rangle =0 \ (\beta \in F \setminus \Delta )\}.$$
	Then we have the following lemma.
	\begin{lem}[\cite{HKT}]\label{lem:cdelta}
			(1)\ For $\Delta_{1} \subset F$\ 
			$$\overline{\mathcal{C}^{\Delta_{1}}} = \bigcup_{\Delta \subset \Delta_{1}} \mathcal{C}^{\Delta}$$
			is a disjoin union. In particular $\overline{\mathcal{C}} = \bigcup_{\Delta \subset F} \mathcal{C}^{\Delta}$ is a disjoint union. 
			
			(2)\ $\Delta_{1} \subset \Delta_{2}$ if and only if $ \mathcal{C}^{\Delta_{1}} \subset \overline{\mathcal{C}^{\Delta_{2}}}$, for $\Delta_{1}, \Delta_{2} \subset F$. 
	\end{lem}
	For each $\alpha \in F$ we define $H_{\alpha} \in \mathfrak{a}$ by 
	$$\langle H_{\alpha}, \beta \rangle = \delta_{\alpha \beta} \ (\beta \in F),$$
	where $\delta_{\alpha \beta}$ is Kronecker's delta.
	Then for $\Delta \subset F$ we have 
	$$\mathcal{C}^{\Delta} = \left\{ \left.\sum_{\alpha \in \Delta} x_{\alpha} H_{\alpha}\ \right| \ x_{\alpha} >0\right\}.$$

\section{Construction of retractions} 
	The notation of the preceding section will be preserved.
	The linear isotropy representation of a Riemannian symmetric space $G/K$ is called an $s$-representation. 
	The $s$-representation of $G/K$ on $T_{o}(G/K)$ and the adjoint representation $\mathrm{Ad}(K)$ on $\mathfrak{m}$ are equivalent. 
	Since an $s$-representation is an orthogonal representation, for a unit vector $H\in \mathfrak{m}$, the orbit $\mathrm{Ad}(K)H$ is a submanifold of the unit sphere $S \subset \mathfrak{m}$.
	Orbits of $s$-representations are called R-spaces.
	The orbit space of an $s$-representation is homeomorphic to $\overline{\mathcal{C}} $, 
	more precisely for any $X \in \mathfrak{m}$, there exists $k \in K$ and unique $H \in \overline{{\mathcal C}}$ such that $X=\mathrm{Ad}(k)H$.
	The decomposition of $\overline{\mathcal{C}}$ in Lemma \ref{lem:cdelta} is the decomposition of the orbit type.
	From the following theorem, we can see that for each orbit type, there exists a unique minimal orbit.
	\begin{thm}[\cite{HTST}]\label{thm:minimalorb}
	For any nonempty subset $\Delta \subset F$, there exists a unique $H \in S \cap \mathcal{C}^{\Delta}$ such that the linear isotropy orbit $\mathrm{Ad}(K)H$ is a minimal orbit of $S$.
	\end{thm}
	\begin{cor}
	An isolated orbit (i.e. $\Delta = \{\alpha\}$) is a minimal submanifold of $S$.
	\end{cor}
	Kitagawa and Ohnita  (\cite{KO}) calculated the mean curvature vector $m_{H}$ of $\mathrm{Ad}(K)H$ in $\mathfrak{m}$ at $H$:
	$$m_{H}= - \sum_{\lambda \in \tilde{R}_{+}\setminus \tilde{R}_{+}^{\Delta}} \frac{\overline{\lambda}}{\langle \lambda , H\rangle}.$$
	This expression is used in the proof of Theorem \ref{thm:minimalorb}.
	We consider cones over minimal embeddings of R-spaces that obtained in this way, and construct retractions.
	\begin{lem}[\cite{HKT}]\label{Lem:3}
	Suppose $\phi$ is a mapping of \ $\overline{\mathcal{C}}$ into itself such that $\phi (\mathcal{C}^{\Delta}) \subset \overline{\mathcal{C}^{\Delta}}$ for each $\Delta \subset F$.
	Then $\phi$ extends to a mapping $\Phi$ of $\mathfrak{m}$ as 
	$$\Phi (X) =\mathrm{Ad}(k)\phi(H)$$
	for each $X=\mathrm{Ad}(k)H\ (k \in K, H \in \overline{\mathcal{C}})$.
	\end{lem}
	The following theorem is a generalization of Proposition 2.6 in \cite{HKT}.
	\begin{thm}\label{thm:retr}
	For $A\in \overline{\mathcal{C}}$, we put $\Delta_{0}=\{\alpha \in F \mid \langle \alpha , A \rangle >0\}$.
	Let $f: \overline{\mathcal{C}} \to \mathbb{R}_{\geq 0}$ be a continuous function. 
	Define a continuous mapping $\phi:\overline{\mathcal{C}} \to \{tA \mid t \geq 0\}$ by $\phi (x)=f(x)A$. If $f$ satisfies 
\begin{enumerate}
	\item[(1)] $f(tA)=t \ (t\geq 0)$,
	\item[(2)] $f|_{\mathcal{C}^{\Delta}}=0\ (\Delta \subset F \ \text{with}\  \Delta_{0} \not\subset \Delta)$,
\end{enumerate}
	then $\phi$ extends to a retraction $\Phi : \mathfrak{m} \to C_{\mathrm{Ad}(K)A}$.
	\end{thm}
	\begin{proof}
	First, we show that $\phi$ satisfies the assumption of Lemma \ref{Lem:3}.
	For $\Delta \subset F$ if $\Delta_{0} \subset \Delta$, then $\mathcal{C}^{\Delta_{0}} \subset \overline{\mathcal{C}^{\Delta}}$.
	Hence
	$$\phi (\mathcal{C}^{\Delta}) =\{tA \mid t\geq 0\} \subset \mathcal{C}^{\Delta_{0}} \subset \overline{\mathcal{C}^{\Delta}}$$
	holds.
	If $\Delta_{0} \not\subset \Delta$, then $\phi (\mathcal{C}^{\Delta})=\{0\}$ since $f|_{\mathcal{C}^{\Delta}}=0$.
	Therefore, $\phi$ satisfies the assumption of Lemma \ref{Lem:3}.
	We also get 
	\begin{eqnarray*}
	\Phi(\mathfrak{m})&=& \{\mathrm{Ad}(k)f(H)A \mid k\in K , H\in \overline{\mathcal{C}}\}\\
					  &=& \{t \mathrm{Ad}(k)A \mid k \in K,\ t\geq 0 \} =C_{\mathrm{Ad}(K)A}.
	\end{eqnarray*}
	Thus $\Phi$ is a surjection from $\mathfrak{m}$ onto $C_{\mathrm{Ad}(K)A}$.
	Next we show that $\Phi$ is continuous.
	Let $\{P_{n}\}_{n \in \mathbb{N}}$ be a sequence in $\mathfrak{m}$ with limit $P_{\infty}\in \mathfrak{m}$.
	Points $P_{n}$ and $P_{\infty}$ can be expressed as 
	$P_{n}=\mathrm{Ad}(k_{n})H_{n},\ P_{\infty}=\mathrm{Ad}(k_{\infty})H_{\infty}$ where $k_{n}, k_{\infty} \in K$ and $H_{n}, H_{\infty} \in \overline{\mathcal{C}}$.
	Since the projection $ \mathfrak{m} \to \overline{\mathcal{C}}; X= \mathrm{Ad}(k)H \mapsto H$ is continuous, we have 
	$\lim_{n\to \infty}H_{n}=H_{\infty}$. 
	We put $\Delta_{\infty}=\{\alpha \in F \mid \langle \alpha , H_{\infty} \rangle >0 \}$, 
	$Z_{K}^{H_{\infty}}=\{k \in K \mid \mathrm{Ad}(k)H_{\infty} =H_{\infty}\}$ and 
	$Z_{K}^{\Delta_{\infty}}=\{k \in K \mid \mathrm{Ad}(k)|_{\overline{\mathcal{C}^{\Delta_{\infty}}}} =\mathrm{id} \}$.
	Since $Z_{K}^{H_{\infty}}=Z_{K}^{\Delta_{\infty}}$ (\cite{HKT}), 
	for any accumulation point $\tilde{k} \in K$ of $\{k_{n} \}_{n \in \mathbb{N}}$, 	
	$\mathrm{Ad}(\tilde{k})|_{\overline{\mathcal{C}^{\Delta_{\infty}}}}=\mathrm{Ad}(k_{\infty})|_{\overline{\mathcal{C}^{\Delta_{\infty}}}}$.
	Thus, we have $\lim_{n \to \infty} \mathrm{Ad}(k_{n})|_{\overline{\mathcal{C}^{\Delta_{\infty}}}}=\mathrm{Ad}(k_{\infty})|_{\overline{\mathcal{C}^{\Delta_{\infty}}}} $.
	Therefore 
	$$\lim_{n \to \infty} \Phi(P_{n})=\lim_{n \to \infty}\mathrm{Ad}(k_{n})f(H_{n})A=\mathrm{Ad}(k_{\infty})f(H_{\infty})A=\Phi(P_{\infty})$$
	Hence $\Phi$ is a retraction from $\mathfrak{m}$ onto $C_{\mathrm{Ad}(K)A}$.
	\end{proof}

	\begin{prop}\label{C1}
	Let $\Phi : \mathfrak{m} \to C_{\mathrm{Ad}(K)A}$ be a retraction which constructed by Theorem \ref{thm:retr}.
	If $\Phi|_{\mathfrak{a}\setminus \Phi^{-1}(\{0\})}$ is $C^{1}$, then $\Phi|_{\mathfrak{m}\setminus \Phi^{-1}(\{0\})}$ is $C^{1}$.
	In this case $\Phi$ is area-nonincreasing  if and only if $J(d\Phi)_{x} \leq 1$ holds for each $x \in \mathcal{C}\setminus \Phi^{-1}(\{0\})$.
	\end{prop}
\begin{proof}
	If $\Phi$ is $C^{1}$ at $H \in \overline{\mathcal{C}}$, then $\Phi$ is $C^{1}$ at $\mathrm{Ad}(k)H$ for all $k \in K$.
	Thus we assume $H \in \overline{\mathcal{C}}\setminus \Phi^{-1}(\{0\})$.
	For $H \in \overline{\mathcal{C}}\setminus \Phi^{-1}(\{0\})$, we put $\Delta =\{\alpha \in F \mid  \langle \alpha , H\rangle >0 \}$. 
	Since $f(H)>0$, we get $\Delta_{0} \subset \Delta $ and $\mathcal{C}^{\Delta_{0}} \subset \overline{\mathcal{C}^{\Delta}}$.
	By Lemma \ref{lem:decomp}, we have 
	$$\mathfrak{m} = \mathfrak{a}+ \sum_{\lambda \in \tilde{R}_{+} \setminus \tilde{R}_{0}} \mathbb{R}\cdot T_{\lambda}.$$
	Since $\Phi|_{\mathfrak{a}\setminus \Phi^{-1}(\{0\})}$ is $C^{1}$, we consider only $T_{\lambda}$ direction for each $\lambda \in \tilde{R}_{+} \setminus \tilde{R}_{0}$.
	If $\langle \lambda ,H \rangle=0$, then $[T_{\lambda},H ] = \langle \lambda ,H \rangle S_{\lambda} =0$ from Lemma \ref{lem:decomp}.
	Thus there exists $k \in Z_{K}^{H} =\{k \in K \ | \ {\rm Ad}(k)H=H \}$ such that  $\mathrm{Ad}(k)T_{\lambda} \in \mathfrak{a}$.
	Therefore 
	\begin{eqnarray*}
	\Phi (H+t T_{\lambda})= \mathrm{Ad}(k)^{-1} \Phi (\mathrm{Ad}(k)(H+tT_{\lambda})).
	\end{eqnarray*}
	Since $\mathrm{Ad}(k)(H+tT_{\lambda}) \in \mathfrak{a}$ and $\Phi|_{\mathfrak{a}\setminus \Phi^{-1}(\{0\})}$ is $C^{1}$,
	we have the directional derivative of $\Phi$ along $T_{\lambda}$. 
	If $\langle \lambda ,H \rangle \neq 0$, then from Lemma \ref{lem:decomp} we have
	that $c(t)=\mathrm{Ad}\left( \exp(-tS_{\lambda}/\langle \lambda ,H \rangle )\right)H$
	is curve in $\mathfrak{m}$ with $c(0)=H$ and $c'(0)=T_{\lambda}$. 
	Thus
	\begin{eqnarray*}
	\left. \frac{d}{dt}\right|_{t=0} \Phi (c(t))&=& \left. \frac{d}{dt}\right|_{t=0} \Phi \left( \mathrm{Ad}\left( \exp\frac{-tS_{\lambda} }{\langle \lambda ,H \rangle} \right)H \right)\\
	&=&\frac{[-S_{\lambda} ,\phi (H) ]}{\langle \lambda ,H \rangle}= \frac{\langle \lambda ,A \rangle}{\langle \lambda ,H \rangle}f(H)T_{\lambda}.
	\end{eqnarray*}
	Therefore $\Phi$ is a differentiable retraction from $\mathfrak{m}$ into $C_{\mathrm{Ad}(K)A}$.
	Since $\Phi|_{\mathfrak{m}\setminus \Phi^{-1}(\{0\})}$ is $C^{1}$, the mapping $\overline{\mathcal{C}}\setminus \Phi^{-1}(\{0\}) \to \mathbb{R} ;x \mapsto J(d\Phi_{x})$ is continuous.
	Hence, if $J(d\Phi_{x}) \leq 1 \ (x\in \mathcal{C} \setminus \Phi^{-1}(\{0\}))$, then 
	$J(d\Phi_{x}) \leq 1 \ (x\in \overline{\mathcal{C}} \setminus \Phi^{-1}(\{0\}))$.
	\end{proof}
	
	We will compute $J(d\Phi_{x})$ of $\Phi$ in Theorem \ref{thm:retr} for $x \in \mathcal{C} \setminus \Phi^{-1}(\{0\})$.
	
	\begin{prop}\label{prop:jacobian}
	We denote $R_{+}^{\Delta_{0}} = \{ \lambda \in R_{+} \mid \langle \lambda , A \rangle=0 \}$.
	$$J(d\Phi_{x})=\|(\mathrm{grad}f)_{x}\| \prod_{\lambda \in R_{+}\setminus R_{+}^{\Delta_{0}}} \left( \frac{\langle \lambda ,A  \rangle}{\langle \lambda ,x  \rangle} f(x)\right)^{m(\lambda)} \ (x \in \mathcal{C} \setminus \Phi^{-1}(\{0\})).$$
	\end{prop}
	\begin{proof}
	From the proof of Proposition \ref{C1}, we have 
	
$$
	d\Phi_{x}(H) =df_{x}(H)A  \ (H \in \mathfrak{a}),\quad
	d\Phi_{x}(T_{\lambda})=\frac{\langle \lambda , A \rangle}{\langle \lambda , x\rangle}f(x)T_{\lambda}\ (\lambda \in \tilde{R}_{+} \setminus \tilde{R}_{0})
$$
	for $x \in \mathcal{C} \setminus \Phi^{-1}(\{0\})$.
	Thus we get 
	$$d\Phi_{x}(\mathfrak{a})\subset \mathbb{R}A \subset \mathfrak{a},\quad d\Phi_{x}\left( \sum_{\mu \in R_{+}}\mathfrak{m}_{\mu}\right) \subset \sum_{\mu \in R_{+}}\mathfrak{m}_{\mu}.$$
	Since $\mathfrak{a}$ and $\sum_{\mu \in R_{+}}\mathfrak{m}_{\mu}$ are orthogonal, 
	we have 
	$$J(d\Phi_{x})=J(d\Phi_{x}|_{\mathfrak{a}})\times J(d\Phi_{x}|_{\sum_{\mu\in R_{+}}\mathfrak{m}_{\mu}} ).$$
	We put $J_{1}(x)=J(d\Phi_{x}|_{\mathfrak{a}}), J_{2}(x)=J(d\Phi_{x}|_{\sum_{\mu\in R_{+}}\mathfrak{m}_{\mu}})$
	and compute each of these.
	\begin{eqnarray*}
	J_{1}(x)&=&\sup \{\|d\Phi_{x}(v)\| \mid v \in \mathfrak{a}, \|v\|=1 \}\\
	&=&\sup \{\langle (\mathrm{grad}f)_{x},  v\rangle  \mid v \in \mathfrak{a}, \|v\|=1 \} =\| (\mathrm{grad}f)_{x}\|.
	\end{eqnarray*}
	Since $\ker \left( d\Phi_{x}|_{\sum_{\mu\in R_{+}}\mathfrak{m}_{\mu}}\right) =\sum_{\mu \in R_{+}^{\Delta_{0}}} \mathfrak{m}_{\mu}$,
	$\{T_{\lambda} \mid \lambda \in \tilde{R}_{+}, \langle \lambda , A\rangle >0\}$ is an orthonormal basis of 
	$\ker \left( d\Phi_{x}|_{\sum_{\mu\in R_{+}}\mathfrak{m}_{\mu}}\right)^{\perp}=\sum_{\mu \in R_{+}\setminus R_{+}^{\Delta_{0}}}\mathfrak{m}_{\mu}$.
	Hence
		\begin{eqnarray*}
	J_{2}(x)&=& \left\| \bigwedge _{\lambda \in \tilde{R}_{+},  \langle \lambda , A \rangle >0 } d\Phi_{x}(T_{\lambda }) \right\|=
\left\| \bigwedge _{\lambda \in \tilde{R}_{+},  \langle \lambda , A \rangle >0 } \frac{\langle \lambda , A\rangle }{\langle \lambda , x\rangle} f(x)T_{\lambda } \right\|\\
	        &=& \prod_{\lambda \in \tilde{R}_{+},  \langle \lambda , A \rangle >0 } \frac{\langle \lambda , A\rangle }{\langle \lambda , x\rangle} f(x) = \prod_{ \lambda \in R_{+}\setminus R_{+}^{\Delta_{0}}} \left(  \frac{\langle \lambda , A\rangle }{\langle \lambda , x\rangle} f(x) \right)^{m(\lambda)}.
	\end{eqnarray*}
	Therefore we get
	$$J(d \Phi)_{x}= J_{1}(x) J_{2}(x)= \|({\rm grad}f)_{x} \| \prod_{\lambda \in R_{+} \setminus R_{+}^{\Delta_{0}}} \left( \frac{\langle \lambda , A\rangle }{\langle \lambda , x\rangle }f(x)\right) ^{m(\lambda)} .$$
	\end{proof}
	
\section{Example of area-minimizing cones over R-spaces}\label{eg} 
	Using Theorem \ref{thm:retr}, Proposition \ref{C1} and Proposition \ref{prop:jacobian}, we investigate area-minimizing properties of cones over R-spaces. 
	First we consider cones over isolated orbits of $s$-representations of irreducible symmetric pairs of compact type  of rank two.
Principal orbits of these representations are homogeneous  hypersurfaces in the sphere.
The area-minimizing properties of the cones over homogeneous minimal hypersurfaces were investigated in \cite{Hs} and \cite{L}.

We shall follow the notations of root systems in \cite{Bo}.
	Partly we used Maxima\footnote{http://maxima.sourceforge.net/} for algebraic computations.
	
\subsection{Type ${\rm A}_{2}$.} 
	$$\mathfrak{a}=\{\xi_{1} e_{1}+ \xi_{2} e_{2}+ \xi_{3} e_{3} \mid \xi_{1}+ \xi_{3}+\xi_{3}=0 \},$$
	$$F=\{\alpha_{1} = e_{1}-e_{2} , \alpha_{2}= e_{2}-e_{3}\}.$$
	Then we have $R_{+}=\{\alpha_{1} , \alpha_{2}, \alpha_{1}+\alpha_{2}\}$. For $\lambda \in R_{+}$, we put $m=m(\lambda)$. 
	We have 
	$$
	H_{\alpha_{1}}=\frac{1}{3}(2e_{1}-e_{2}-e_{3}),\quad H_{\alpha_{2}}=\frac{1}{3}(e_{1}+e_{2}-2e_{3}).
	$$
	We put 
	$$
	A_{1}=\frac{H_{\alpha_{1}}}{\|H_{\alpha_{1}}\|}=\frac{1}{\sqrt{6}}(2e_{1}-e_{2}-e_{3}),\quad A_{2}=\frac{H_{\alpha_{2}}}{\|H_{\alpha_{2}}\|}=\frac{1}{\sqrt{6}}(e_{1}+e_{2}-2e_{3}).
	$$
	Since $\mathrm{Ad}(K)A_{1}$ and $\mathrm{Ad}(K)A_{2}$ are isometric,
	we consider only the cone over $\mathrm{Ad}(K)A_{1}$.
\subsubsection{Cones over $\mathrm{Ad}(K)A_{1}$.} 
	We put $\Delta_{0} = \{\alpha_{1}\}$ then $R_{+}^{\Delta_{0}}=\{\alpha_{2}\}$.
	For $x= x_{1}H_{\alpha_{1}}+x_{2}H_{\alpha_{2}} \in \overline{\mathcal{C}}$, we define 
	$$
	f(x)=\sqrt{\frac{2}{3}}\left(\langle \alpha_{1} ,x \rangle ^{2} \left\langle  \alpha_{1} +\frac{3}{2}\alpha_{2} ,x \right\rangle\right)^{\frac{1}{3}} =\sqrt{\frac{2}{3}}\left( x_{1}^{2} \left(x_{1} +\frac{3}{2} x_{2}\right) \right)^{\frac{1}{3}}.
	$$
	Since 
\begin{enumerate}
	\item[(1)] $f(tA_{1})=\sqrt{\frac{2}{3}}\left( \left(\sqrt{\frac{3}{2}}t\right)^3 \right)^{\frac{1}{3}}=t$,
	\item[(2)] for each $\Delta \subset F$, if $\Delta_{0} \not\subset \Delta$, then $f|_{\mathcal{C}^{\Delta}}=0$, 
\end{enumerate}
	we can apply Theorem \ref{thm:retr} to this case.
	It is clear that $\Phi|_{\mathfrak{a}\setminus \Phi^{-1}(\{0\})}$ is $C^{1}$. 
	Thus $\Phi$ is a differentiable retraction by Proposition \ref{C1}.
Since
	\begin{eqnarray*}
			\frac{\partial f}{\partial x_{1}}(x)&=&  \sqrt{\frac{2}{3}} \left( x_{1}^{2} \left( x_{1} + \frac{3}{2} x_{2} \right) \right)^{- \frac{2}{3}} (x_{1}^{2}+x_{1}x_{2}),\\
			\frac{\partial f}{\partial x_{2}}(x)&=&  \sqrt{\frac{2}{3}} \left( x_{1}^{2} \left( x_{1} + \frac{3}{2} x_{2} \right) \right)^{- \frac{2}{3}} \frac{x_{1}^2}{2},
	\end{eqnarray*}
	we get 
	$$
	J_{1}(x)=\|({\rm grad }f)_{x} \| = \sqrt{\frac{2}{3}} \left( x_{1}^{2} \left( x_{1} + \frac{3}{2} x_{2} \right) \right)^{- \frac{2}{3}} \sqrt{\frac{3}{2}x_{1}^{4} +3 x_{1}^{3} x_{2} + 2x_{1}^{2}x_{2}^{2} }.
	$$	
	On the other hand,
	$$
			J_{2}(x)= \left(\frac{\langle \alpha _{1} , A_{1} \rangle}{\langle \alpha _{1} , x \rangle}f(x) \right) ^{m} 
                        \left(\frac{\langle \alpha _{1}+\alpha _{2} , A_{1} \rangle}{\langle \alpha _{1}+\alpha _{2} , x \rangle}f(x) \right) ^{m}
					= \left( \frac{\left( x_{1} \left( x_{1} +\frac{3}{2}x_{2} \right) ^2 \right)^{ \frac{1}{3}} } { x_{1}+x_{2} } \right) ^{m}.
	$$	
	Then 
	$$(x_{1}+x_{2})^3 -x_{1}\left( x_{1}+ \frac{3}{2} x_{2} \right)^2= \frac{3}{4}x_{1}x_{2}^{2} +x_{2}^{3} \geq 0,$$
	thus 
	$$\left( \frac{\left( x_{1} \left( x_{1} +\frac{3}{2}x_{2} \right) ^2 \right)^{ \frac{1}{3}} } { x_{1}+x_{2} } \right) \leq 1.$$
	We put 
	$$D= J_{1}(x)\times \left( \frac{\left( x_{1} \left( x_{1} +\frac{3}{2}x_{2} \right) ^2 \right)^{ \frac{1}{3}} } { x_{1}+x_{2} } \right)^2 =\left( \frac{(3x_{1}^2 + 6x_{1}x_{2} +4 x_{2}^{2} )^3 x_{1}^{2} (2 x_{1}+ 3 x_{2})^4}{3^3 2^4 (x_{1}+x_{2})^{12}}\right)^{\frac{1}{6}}.$$
	Since 
	$$J(d \Phi )_{x}= D\times \left( \frac{\left( x_{1} \left( x_{1} +\frac{3}{2}x_{2} \right) ^2 \right)^{ \frac{1}{3}} } { x_{1}+x_{2} } \right)^{m-2},$$
	if $D \leq 1$, then $J(d\Phi_{x}) \leq 1$ for $m\geq 2$.
	Since 
	\begin{eqnarray*}
			&& 3^3 2^4 (x_{1}+x_{2})^{12}- (3x_{1}^2 + 6x_{1}x_{2} +4 x_{2}^{2} )^3 x_{1}^{2} (2 x_{1}+ 3 x_{2})^4\\
			&=&216 x_{1}^{10}x_{2}^{2}+2376x_{1}^{9}x_{2}^{3}+11925x_{1}^{8}x_{2}^{4}+35838x_{1}^{7}x_{2}^{5}+71120x_{1}^{6}x_{2}^{6}\\ 
			&&+96888x_{1}^{5}x_{2}^{7}+91152x_{1}^{4}x_{2}^{8}+57888x_{1}^{3}x_{2}^{9}+23328x_{1}^{2}x_{2}^{10}+5184x_{1}x_{2}^{11}+432x_{2}^{12}\\
			&\geq &0,
	\end{eqnarray*}
	we have $D\leq 1$.
	Therefore, cones over $\mathrm{Ad}(K)A_{1}$ are area-minimizing for $m \geq 2$.
	
\subsection{Types ${\rm B}_{2}$, ${\rm BC}_{2}$ and ${\rm C}_{2}$.}
	Types ${\rm C}_{2}$ and ${\rm B}_{2}$ are isomorphic, thus it suffices to compute the type ${\rm B}_{2}$ case.
	Moreover setting the multiplicity of long roots to zero, the set of restricted roots of type ${\rm BC}_{2}$ reduces to that of type ${\rm B}_{2}$.
	We have 
	$$F =\{\alpha _{1} =e_{1}-e_{2} ,\ \alpha _{2}= e_{2} \},$$ 
		$$R_{+}=\{\alpha _{1} ,\ \alpha _{2} ,\ \alpha _{1}+\alpha _{2} ,\ \alpha_{1}+2\alpha_{2},\ 2\alpha_{1}+2\alpha_{2},\ 2\alpha_{2}  \},$$
		$$H_{\alpha_{1}}= e_{1},\quad H_{\alpha_{2}}=e_{1}+e_{2},$$
	and put 
		$$m(\alpha_{1})=m_{1} ,\ m(\alpha_{2})=m_{2} ,\ m(2\alpha _{2})=m_{3}.$$
\subsubsection{Cones over ${\rm Ad}(K)A_{1}$.} 
		We put $\Delta_{0}=\{\alpha_{1} \}$, then we have  
		$$A_{1}=\frac{H_{\alpha_{1}}}{\|H_{\alpha_{1}}\|}= e_{1},$$
		and 
		$$R_{+}^{\Delta_{0}}=\{\lambda \in R_{+} \ | \ \langle \lambda , A_{1}\rangle =0\}=\{\alpha _{2} , 2\alpha_{2}\}.$$ 
		For $x=x_{1}H_{\alpha_{1}} +x_{2}H_{\alpha_{2}}\in \overline {\mathcal{C}}$, we define 
		$$f(x)= \sqrt{\langle \alpha_{1},x \rangle \langle \alpha_{1}+ 2\alpha _{2},x \rangle}=\sqrt{x_{1}(x_{1}+2x_{2})}.$$
		Then we can show that $f$ satisfies the condition of Theorem \ref{thm:retr} and $\Phi$ is differentiable. Moreover $J(d\Phi_{x}) \leq 1$ holds for $m_{2}+m_{3} \geq 2$.

	Therefore, cones over $\mathrm{Ad}(K)A_{1}$ are area-minimizing for $m_{2}+m_{3} \geq 2$.
	
\subsubsection{Cones over ${\rm Ad}(K)A_{2}$.}
		We put $\Delta_{0}=\{\alpha_{2} \}$, then we have 
		$$A_{2}=\frac{H_{\alpha_{2}}}{\|H_{\alpha_{2}}\|}= \frac{e_{1}+e_{2}}{\sqrt{2}}$$
		and 
		$$R_{+}^{\Delta_{0}}=\{\lambda \in R_{+} \ | \ \langle \lambda , A_{2}\rangle =0\}=\{\alpha _{1} \}.$$ 
		For $x= x_{1}H_{\alpha_{1}}+ x_{2}H_{\alpha_{2}} \in \overline{\mathcal{C}}$, we define 
		$$f(x)= \sqrt{2} \left( \langle \alpha_{2} , x \rangle^2 \left\langle \frac{3}{2} \alpha_{1}+ \alpha_{2} , x \right\rangle \right)^{\frac{1}{3}}=\sqrt{2} \left( x_{2}^2 \left(\frac{3}{2}x_{1} +x_{2} \right)\right)^{\frac{1}{3}}.$$
		Then we can show that $f$ satisfies the condition of Theorem \ref{thm:retr} and $\Phi$ is differentiable. Moreover $J(d\Phi_{x}) \leq 1$ holds for $m_{2}+m_{3} \geq 2$.

	Therefore, cones over $\mathrm{Ad}(K)A_{2}$ are area-minimizing for $m_{2}+m_{3}\geq 2$.

	\subsection{Type ${\rm G}_{2}$.}
		We have 
		$$F =\{\alpha _{1} ,\ \alpha _{2}\},$$ 
		$$R_{+}=\{\alpha _{1} ,\ \alpha _{2} ,\ \alpha _{1}+\alpha _{2} ,\ 2\alpha _{1}+\alpha _{2} ,\ 3\alpha _{1}+\alpha _{2} ,\ 3\alpha _{1}+2\alpha _{2}  \},$$
		$$\langle \alpha _{1} , \alpha_{1} \rangle =1,\ \langle \alpha _{1} , \alpha_{2} \rangle = -\frac{3}{2},\ \langle \alpha _{2} , \alpha_{2} \rangle =3,$$
		$$H_{\alpha_{1}}= 4\alpha _{1} +  2\alpha_{2} ,\quad H_{\alpha_{2}}= \frac{2}{3} (3 \alpha_{1} +2 \alpha _{2}),$$
		and put
		$$m=m(\alpha_{1}) = m(\alpha_{2}).$$
	 
\subsubsection{Cones over ${\rm Ad}(K)A_{1}$.}
		We put $\Delta_{0}=\{\alpha_{1} \}$ then we have  
		$$A_{1}=\frac{H_{\alpha_{1}}}{\|H_{\alpha_{1}}\|}$$
		and 
		$$R_{+}^{\Delta_{0}}=\{\lambda \in R_{+} \ | \ \langle \lambda , A_{1}\rangle =0\}=\{\alpha _{2} \}.$$
		For $x=x_{1}H_{\alpha_{1}} +x_{2}H_{\alpha_{2}} \in \overline{\mathcal{C}}$, we define 
		$$f(x)= \sqrt{4\langle \alpha_{1},x \rangle \langle \alpha_{1}+ \alpha _{2},x \rangle}=\sqrt{4x_{1}(x_{1}+x_{2})}.$$
		Then we can show that $f$ satisfies the condition of Theorem \ref{thm:retr} and $\Phi$ is differentiable. Moreover $J(d\Phi_{x}) \leq 1$ holds for $m \geq 2$.

	Therefore cones over $\mathrm{Ad}(K)A_{1}$ are area-minimizing for $m \geq 2$. 

\subsubsection{Cones over ${\rm Ad}(K)A_{2}$.}
		We put $\Delta_{0}=\{\alpha_{2} \}$ then we have 
		$$A_{2}=\frac{H_{\alpha_{2}}}{\|H_{\alpha_{2}}\|},$$
		and
		$$R_{+}^{\Delta_{0}}=\{\lambda \in R_{+} \ | \ \langle \lambda , A_{1}\rangle =0\}=\{\alpha _{1} \}.$$ 
		For $x=x_{1}H_{\alpha_{1}}+ x_{2}H_{\alpha_{2}} \in \overline{\mathcal{C}}$, we define
		$$f(x)= \sqrt{\frac{4}{3}\langle \alpha_{2},x \rangle \langle 3 \alpha_{1}+ \alpha _{2},x \rangle}=\sqrt{\frac{4}{3}x_{2}(3 x_{1}+x_{2})}.$$
		Then we can show that $f$ satisfies the condition of Theorem \ref{thm:retr} and $\Phi$ is differentiable. Moreover $J(d\Phi_{x}) \leq 1$ holds for $m \geq 2$.

		Therefore, cones over $\mathrm{Ad}(K)A_{2}$ are area-minimizing for $m \geq 2$. 

\bigskip		
	By the above computation, we get the following table of cones over isolated orbits of the $s$-representations of irreducible symmetric spaces of rank two.

		\begin{landscape}
	\hspace{-24pt} \begin{tabular}{|c|c|c|c|cc|c|cc|}
\hline
type      &    symmetric pair                              &  multiplicities& $A_{i}$&orbit                   &symm. or not &$\substack{{\rm dim.\ of\ orbit} \\ {\rm and\ sphere}}$& area-min.&\\ \hline \hline
${\rm A}_{2}$  & $({\rm SU}(3), {\rm SO}(3))$                   & $(1, 1)$  & $A_{1}$&  $\mathbb{R}P^2 $      & \text{symmetric}&$(2,4)$&            &\\ 
			   & $({\rm SU}(3)\times {\rm SU}(3), {\rm SU}(3))$ & $(2, 2)$  & $A_{1}$&  $\mathbb{C}P^2 $     &\text{symmetric}&$(4,7)$& $\bigcirc$  &\cite{Ke} \\
			   & $({\rm SU}(6), {\rm Sp}(3))$                   & $(4, 4)$  & $A_{1}$&  $\mathbb{H}P^2 $  &\text{symmetric}&$(8,13)$& $\bigcirc$  &\cite{Ka}\\
			   & $(E_{6},F_{4})       $                         & $(8, 8)$  & $A_{1}$&  $\mathbb{O}P^2 $  &\text{symmetric}&$(16,25)$& $\bigcirc$  &\\ \hline
${\rm B}_{2}$  & $({\rm SO}(5)\times {\rm SO}(5), {\rm SO}(5))$ & $(2, 2)$  & $A_{1}$&  $\widetilde{G_{2}(\mathbb{R}^{5})} $ &\text{symmetric}&$(6,9)$&$\bigcirc$&\cite{HKT}\\ 
			   &                                                &           & $A_{2}$&  ${\rm SO}(5)/{\rm U}(2) $&&$(6,9)$& $\bigcirc$&\cite{Ke} \\
			   & $({\rm SO}(5),{\rm SO}(2)\times {\rm SO}(3))$  & $(1, 1)$  & $A_{1}$&                        &\text{symmetric}&$(3,5)$& &\\
			   &                                                &           & $A_{2}$&                         &&$(3,5)$& &\\ 
			   & $({\rm SO}(4+n),{\rm SO}(2)\times {\rm SO}(2+n))$& $(1, n)$& $A_{1}$&                         &\text{symmetric}&$(n+2,2n+3)$& $\bigcirc \ (n\geq 2)$&\cite{HKT}\\
			   &                                                &           & $A_{2}$&                         &&$(2n+1,2n+3)$& $\bigcirc \ (n\geq 2)$&\\ \hline
${\rm C}_{2}$  &$({\rm Sp}(2), {\rm U}(2))$                     & $(1,1)$   & $A_{1}$&                         & &$(3,5)$&&\\
			   &                                                &           & $A_{2}$&  ${\rm U}(2)/{\rm O}(2) $&\text{symmetric}&$(3,5)$&&\\
			   &$({\rm Sp}(2)\times {\rm Sp}(2), {\rm Sp}(2))$  & $(2,2)$   & $A_{1}$&                          &&$(6,9)$&$\bigcirc$ &\\
			   &                                                &           & $A_{2}$&  ${\rm Sp}(2)/{\rm U}(2) $&\text{symmetric}&$(6,9)$&$\bigcirc$ &\\
			   &$({\rm Sp}(4),{\rm Sp}(2)\times {\rm Sp}(2))$   & $(4,3)$   & $A_{1}$&                          &&$(11,15)$&$\bigcirc$ &\\
			   &                                                &           & $A_{2}$&  ${\rm Sp}(2) $&          \text{symmetric}&$(10,15)$&$\bigcirc$& \\
			   &$({\rm SU}(4),{\rm S}({\rm U}(2)\times {\rm U}(2)))$& $(2,1)$&$A_{1}$&                          &&$(5,7)$&$\bigcirc$&\\
			   &                                                &           & $A_{2}$&  ${\rm U}(2)  $          &\text{symmetric}&$(4,7)$&$\bigcirc$&\cite{L} \\
			   &$({\rm SO}(8), {\rm U}(4))$                     & $(4,1)$   & $A_{1}$&  ${\rm U}(4)/({\rm Sp(1)\times {\rm U}(2)}) $&&$(9,11)$&$\bigcirc $&\\
			   &                                                &           & $A_{2}$&  ${\rm U}(4)/{\rm Sp}(2) $ &\text{symmetric}&$(6,11)$&$\bigcirc$&\\ \hline
\end{tabular}
\newpage 
	\begin{tabular}{|c|c|c|c|cc|c|cc|}
\hline
type      &    symmetric pair                              &  multiplicities& $A_{i}$&orbit                               &symm. or not&$\substack{{\rm dim.\ of\ orbit} \\ {\rm and\ sphere}}$& area-min.&\\ \hline \hline
${\rm BC}_{2}$ &$({\rm SU}(4+n),{\rm S}({\rm U}(2)\times {\rm U}(2+n)))$&$(2,(2n,1))$&$A_{1}$&  &&$(2n+3,4n+7)$&$\bigcirc \ (n\geq 1)$ &\\
			   &                                                        &            &$A_{2}$& &&$(4n+4,4n+7)$&$\bigcirc \ (n\geq 1)$ &\\
			   &$({\rm SO}(10),{\rm U}(5))$                     &$(4,(4,1))$&$A_{1}$& ${\rm U}(5)/({\rm Sp}(1)\times {\rm U}(3))$&&$(13,19)$& $\bigcirc$ & \\
			   &                                                &           &$A_{2}$& ${\rm U}(5)/({\rm Sp}(2)\times {\rm U}(1))$&&$(14,19)$& $\bigcirc$&\\
			   &$({\rm Sp}(4+n),{\rm Sp}(2)\times {\rm Sp}(2+n))$&$(4,(4n,3))$&$A_{1}$&                                    &&$(4n+11,8n+15)$&$\bigcirc \ (n\geq 1)$ &\\
			   &                                                 &            &$A_{2}$&                                    &&$(8n+10,8n+15)$&$\bigcirc \ (n\geq 1)$ &\\
			   &$(E_{6},\mathrm{T}^{1}\cdot \mathrm{Spin}(10))$        & $(6,(8,1))$ &$A_{1}$&                                    &&$(21,31)$&$\bigcirc$&\\
			   &                                                &             &$A_{2}$&                                    &&$(24,31)$&$\bigcirc$&\\ \hline
${\rm G}_{2}$  &$(G_{2}, {\rm SO}(4) )$                         & $(1,1)$     &$A_{1}$&                                    &&$(5,7)$&&\\
			   &                                                &             &$A_{2}$&                                    &&$(5,7)$&&\\
			   &$(G_{2} \times G_{2}, G_{2})$                          & $(2,2)$     &$A_{1}$&                               &     &$(10,13)$&$\bigcirc$&\\
			   &                                                &             &$A_{2}$&                                    &&$(10,13)$&$\bigcirc$&\\ \hline
	\end{tabular}
\end{landscape}

	\subsection{Type ${\rm A}_{3}$. }
	Theorem \ref{thm:retr} can be applied to cones over minimal orbits, not only isolated orbits.
	We demonstrate the area-minimizing property for the cone over a minimal orbit, which is not an isolated orbit, of the $s$-representation of symmetric spaces of type ${\rm A}_{3}$.
	$$\mathfrak{a} = \left\{ \sum_{i=1}^{4} \xi_{i} e_{i} \ | \ \sum_{i=1}^{4} \xi_{i}=0 \right\},$$
	$$F =\{\alpha _{1} =e_{1}-e_{2} ,\ \alpha _{2}= e_{2}-e_{3}  ,\ \alpha _{3} =e_{3 } -e_{4}\}.$$
	Then $R_{+}=\{\alpha _{1} , \alpha _{2} ,\alpha _{3} ,\ \alpha _{1}+\alpha _{2},\ \alpha_{2}+\alpha_{3} ,\ \alpha_{1}+\alpha_{2}+\alpha_{3} \}$ 
	and for $\lambda \in R_{+}$, we put $m(\lambda)=m$.
	We have 
	$$H_{\alpha_{1}}= \frac{1}{4}(3e_{1}-e_{2}-e_{3}-e_{4}),\ H_{\alpha_{2}}=\frac{1}{4}(2e_{1}+2e_{2}-2e_{3}-2e_{4}),\ H_{\alpha_{3}} =\frac{1}{4} (e_{1}+e_{2}+e_{3}-3e_{4}).$$
	We put $\Delta_{0}=\{\alpha_{1} ,\alpha_{3}\}$, and we have 
	$$A=\frac{H_{\alpha _{1}}+H_{\alpha_{3}}}{\sqrt{2}} = \frac{e_{1}-e_{4}}{\sqrt{2}}.$$
	Then the orbit $\mathrm{Ad}(K)A$ is a minimal submanifold of the sphere $S \subset \mathfrak{m}$. 
	We get $$R_{+}^{\Delta_{0}}=\{\lambda \in R_{+} \ | \ \langle \lambda , A \rangle =0\}=\{\alpha _{2} \}.$$
	For $x=x_{1}H_{\alpha_{1}}+ x_{2}H_{\alpha_{2}}+x_{3}H_{\alpha_{3}} \in \overline{\mathcal{C}}$, 
	we define
	$$f(x)= \sqrt{2}\left( \langle \alpha_{1}, x \rangle  \langle \alpha_{3}, x \rangle  \langle \alpha_{1}+\alpha_{2}, x \rangle  \langle \alpha_{2}+\alpha_{3}, x \rangle \right)^{\frac{1}{4}}=\sqrt{2} (x_{1}x_{3} (x_{1}+x_{2}) (x_{2}+x_{3}))^{\frac{1}{4}}.$$
	Then we can show that $f$ satisfies the condition of Theorem \ref{thm:retr} and $\Phi$ is differentiable. Moreover $J(d\Phi_{x}) \leq 1$ holds for $m \geq 4$.

	Therefore, cones over $\mathrm{Ad}(K)A$ are area-minimizing for $m \geq 4$.
	The only symmetric pair which satisfies $m\geq 3$ is $({\rm SU}(6), {\rm Sp}(3))$.
	
	\section{Reducible cases}\label{reducible}
	
	In this section, we consider cones over products of two R-spaces.
	Let $(G_{i}, K_{i}) \ (i=1,2)$ be Riemannian symmetric pairs, and put $(G,K)=(G_{1}\times G_{2}, K_{1}\times K_{2})$.
	We define the notation for $(G_{i}, K_{i})$ as follows.
	Let 
	$$\mathfrak{g}_{i}=\mathfrak{k}_{i}+\mathfrak{m}_{i} \ (i=1,2)$$
	be the canonical decompositions of Lie algebras $\mathfrak{g}_{i}$ of $G_{i}$. 
	Take and fix a maximal abelian subspace $\mathfrak{a}_{i}$ in $\mathfrak{m}_{i}$.
	We denote by $R_{i}$ the restricted root system of $(\mathfrak{g}_{i},\mathfrak{k}_{i})$ with respect to $\mathfrak{a}_{i}$.
	We put the fundamental systems $F_{i}$ of $R_{i}$ by $F_{i}=\{\alpha_{i1},\ \ldots , \alpha_{il_{i}} \}$. 
	$R_{i+}$ is the set of positive roots in $R_{i}$.
	We set 
	$$\mathcal{C}_{i} =\{H \in \mathfrak{a}_{i} \mid \langle \alpha , H \rangle >0\ (\alpha \in F_{i})\}, $$
	$$\mathcal{C}_{i}^{\Delta}=\{H \in \mathfrak{a}_{i} \mid \langle \alpha , H \rangle >0\ (\alpha \in \Delta ) ,\ \langle \beta , H \rangle =0\ (\beta \in F_{i} \setminus \Delta)\},$$
	where $\Delta \subset F_{i}$.
	The direct sum of the $s$-representations of $(G_{i}, K_{i})$ is the $s$-representation of $(G,K)=(G_{1}\times G_{2}, K_{1}\times K_{2})$.
	Then, we have 
	$$\overline{\mathcal{C}}=\overline{\mathcal{C}_{1}}\times \overline{\mathcal{C}_{2}}.$$
	For $\Delta \subset F$, $\Delta$ is expressed as $\Delta =\Delta_{1} \cup \Delta_{2}$ where $\Delta_{i} \subset F_{i} \ (i=1,2)$.
	By Theorem \ref{thm:minimalorb} for each $\Delta_{i}$, there exists $A_{i} \in \overline{\mathcal{C}}_{i}$ such that $\mathrm{Ad}(K_{i})A_{i}$ is a minimal orbit of the $s$-representation of $(G_{i},K_{i})$.
	We put $k_{i}=  \dim \mathrm{Ad}(K_{i})A_{i}$ and $k=k_{1}+k_{2}$, then 
	$$A =\sqrt{\frac{k_{1}}{k}}A_{1} +\sqrt{\frac{k_{2}}{k}}A_{2} \in \overline{\mathcal{C}}$$
	is a base point of a minimal orbit of the $s$-representation of $(G,K)$.
	\begin{thm}\label{reduc}
	Let $\Delta_{0} =\Delta_{1} \cup \Delta_{2} \ (\Delta_{i} \subset F_{i})$. We suppose that for the cone over ${\rm Ad}(K_{i})A_{i}$, there exists an area-nonincreasing retraction constructed by 
	a function $f_{i}$ on $\overline{\mathcal{C}_{i}}$ in Theorem \ref{thm:retr}, and that the retraction satisfies 
	$$\prod_{\lambda \in R_{i+} \setminus R_{i+}^{\Delta_{i}}} \left( \frac{\langle \lambda , A_{i} \rangle }{\langle \lambda , x \rangle} f_{i}(x) \right)^{m(\lambda)} \leq 1 \ (x \in \mathcal{C}_{i}).$$
	If $\dim \mathrm{Ad}(K_{i})A_{i} \geq 3$, then there exists an area-nonincreasing retraction $\Phi : \mathfrak{m} \to C_{\mathrm{Ad}(K)A}$ that constructed by 
	some function $f$ on $\overline{\mathcal{C}}$ in Theorem \ref{thm:retr}, and then the retraction satisfies
	$$\prod_{\lambda \in R_{+} \setminus R_{+}^{\Delta_{0}}} \left( \frac{\langle \lambda , A \rangle }{\langle \lambda , x \rangle} f(x) \right)^{m(\lambda)} \leq 1 \ (x \in \mathcal{C}).$$
	\end{thm}
	\begin{proof}
	Let $k_{i}=  \dim \mathrm{Ad}(K_{i})A_{i} ,\ k=k_{1}+k_{2}$ and put $a_{i}=\sqrt{k_{i}/k}$.
	$A =a_{1}A_{1} +a_{2}A_{2}$ holds.
	For $x=(x_{1},x_{2}) \in \overline{\mathcal{C}}_{1}\times \overline{\mathcal{C}}_{2}=\overline{\mathcal{C}}$ we define 
	\begin{eqnarray*}
	f(x)=\left\{
\begin{array}{cl}
\displaystyle \frac{f_{1}(x_{1})f_{2}(x_{2})}{a_{2}^{3}f_{1}(x_{1})+a_{1}^{3}f_{2}(x_{2})}& (f_{1}(x_{1})\neq 0\ \text{or} \ f_{2}(x_{2}) \neq 0)\\
0                                                                           & (f_{1}(x_{1})=f_{2}(x_{2})=0)\\
	\end{array}
	\right.
	\end{eqnarray*}
	We will show that $f$ satisfies the conditions of Theorem \ref{thm:retr}. We can check easily  
	$f(tA)=t$ for $t\geq 0$.
	For $\Delta \subset F$ with $\Delta_{0} \not\subset \Delta$, using $\Delta_{i}' \subset F_{i}$ we can write $\Delta = \Delta_{1}' \cup \Delta_{2}'$. 
	Then $\Delta_{i} \not\subset \Delta_{i}'$ holds $i=1$ or $i=2$. 
	Thus $f_{1}=0$ or $f_{2}=0$ holds on $\mathcal{C}^{\Delta}$.
	Therefore $f|_{\mathcal{C}^{\Delta}}=0$.
	Since $\Phi|_{\mathfrak{a}\setminus (\{0 \})}$ is $C^{1}$, $\Phi$ is a differentiable retraction by Proposition \ref{C1}.
	We calculate $J(d\Phi_{x})$ for $x \in \mathcal{C} \setminus f^{-1}(\{ 0 \})$.
	We put 
	$$J_{1}(x)=\|({\rm grad}f)_{x}\|,\quad J_{2}(x)=\prod_{\lambda \in R_{+}\setminus R_{+}^{\Delta}} \left( \frac{\langle \lambda , A \rangle}{\langle \lambda , x\rangle} f(x)\right)^{m(\lambda)}$$
	for $x=(x_{1},x_{2})=(x_{1}^{1},\ldots , x_{1}^{l_{1}},x_{2}^{1},\ldots , x_{2}^{l_{2}} ) \in \mathcal{C} \setminus f^{-1}(\{ 0 \}) =\mathcal{C}_{1}\times \mathcal{C}_{2}\setminus f^{-1}(\{ 0 \})$.
	Since 
	\begin{eqnarray*}
	\frac{\partial f}{\partial x_{1}^{j}}&=&\frac{\frac{\partial f_{1}}{\partial x_{1}^{j}} a_{1}^{3}f_{2}(x_{2})^{2}}{(a_{2}^{3} f_{1}(x_{1}) +a_{1}^{3}f_{2}(x_{2}))^{2}} \ \ (j \in \{1, \ldots , l_{1}\}),\\
	\frac{\partial f}{\partial x_{2}^{j}}&=&\frac{\frac{\partial f_{2}}{\partial x_{2}^{j}} a_{2}^{3}f_{1}(x_{1})^{2}}{(a_{2}^{3} f_{1}(x_{1}) +a_{1}^{3}f_{2}(x_{2}))^{2}} \ \ (j \in \{1, \ldots , l_{2}\}),
	\end{eqnarray*}
	we get 
		$$({\rm grad}f)_{x} = \frac{ a_{1}^{3}f_{2}(x_{2})^{2} ({\rm grad}f_{1})_{x_{1}} +a_{2}^{3}f_{1}(x_{1})^{2} ({\rm grad}f_{2})_{x_{2}}}{(a_{2}^{3} f_{1}(x_{1}) +a_{1}^{3}f_{2}(x_{2}))^{2}}$$
	and 
	$$J_{1}(x)= \|({\rm grad}f)_{x} \|= \frac{ \sqrt{a_{1}^{6}f_{2}(x_{2})^{4} \|({\rm grad}f_{1})_{x_{1}}\|^2 +a_{2}^{6}f_{1}(x_{1})^{4} \|({\rm grad}f_{2})_{x_{2}}\|^2 }}{(a_{2}^{3} f_{1}(x_{1}) +a_{1}^{3}f_{2}(x_{2}))^{2}}.$$ 
	Since $R_{+}^{\Delta_{0}}=\{\lambda \in R_{+} \mid \langle \lambda , A \rangle =0\} = R_{1+}^{\Delta_{1}} \cup R_{2+}^{\Delta_{2}}$, we get 
	\begin{eqnarray*}
	J_{2}(x)&=& \prod_{\lambda \in R_{+}\setminus R_{+}^{\Delta_{0}}} \left( \frac{\langle \lambda , A \rangle}{\langle \lambda , x\rangle} f(x)\right)^{m(\lambda)}\\
	&=& \prod_{\lambda \in R_{1+}\setminus R_{1+}^{\Delta_{1}}} \left( \frac{\langle \lambda , a_{1}A_{1}\rangle}{\langle \lambda , x_{1}\rangle} f(x)\right)^{m(\lambda)}
	 \prod_{\mu \in R_{2+}\setminus R_{2+}^{\Delta_{2}}} \left( \frac{\langle \mu , a_{2}A_{2}\rangle}{\langle \mu , x_{2}\rangle} f(x)\right)^{m(\mu)}\\
	&=&\prod_{\lambda \in R_{1+}\setminus R_{1+}^{\Delta_{1}}} \left( \frac{\langle \lambda , A_{1}\rangle}{\langle \lambda , x_{1}\rangle}f_{1}(x_{1}) \frac{ a_{1}f(x)}{f_{1}(x_{1})}\right)^{m(\lambda)}
	 \prod_{\mu \in R_{2+}\setminus R_{2+}^{\Delta_{2}}} \left( \frac{\langle \mu , A_{2}\rangle}{\langle \mu , x_{2}\rangle}f_{2}(x_{2})\frac{ a_{2}f(x)}{f_{2}(x_{2}) }\right)^{m(\mu)}.\\
	\end{eqnarray*}
	Put 
	$$J_{2i}(x_{i})=\prod_{\lambda \in R_{i+}\setminus R_{i+}^{\Delta_{i}}} \left( \frac{\langle \lambda , A_{i}\rangle}{\langle \lambda , x_{i}\rangle}f_{i}(x_{i})\right)^{m(\lambda)}\ ,\ J_{1i}(x_{i})=\|({\rm grad}f_{i})_{x_{i}}\| \ (i=1,2).$$
	Note that $J_{2i}(x_{i}) \leq 1, J_{1i}(x_{i})J_{2i}(x_{i}) \leq 1$ holds by assumption.
	Since 
	$$\sum_{\lambda \in R_{i+}\setminus R_{i+}^{\Delta_{i}}} m(\lambda)=\dim {\rm Ad}(K_{i})A_{i}=k_{i},$$
	we can write 
	$$J_{2}(x)=J_{21}(x_{1})J_{22}(x_{2}) \left(\frac{ a_{1}f(x)}{f_{1}(x_{1}) }\right)^{k_{1}}\left(\frac{ a_{2}f(x)}{f_{2}(x_{2}) }\right)^{k_{2}}.$$
	Since $J_{2i}(x_{i}) \leq 1$, 
	$$J_{2}(x) \leq \left(\frac{ a_{1}f(x)}{f_{1}(x_{1}) }\right)^{k_{1}}\left(\frac{ a_{2}f(x)}{f_{2}(x_{2}) }\right)^{k_{2}}.$$
	We put 
	$$X_{1}=\frac{f_{2}(x_{2})}{a_{2}} ,\ X_{2}=\frac{f_{1}(x_{1})}{a_{1}}.$$
	Then we have 
	 $$\left(\frac{ a_{1}f(x)}{f_{1}(x_{1}) }\right)^{k_{1}}\left(\frac{ a_{2}f(x)}{f_{2}(x_{2}) }\right)^{k_{2}}=\frac{X_{1}^{k_{1}}X_{2}^{k_{2}}}{(a_{1}^{2}X_{1} +a_{2}^{2}X_{2})^{k}}.$$
	For $X_{1}, X_{2} >0 $, we define 
	$$\tilde{D}(X_{1},X_{2})=\frac{X_{1}^{k_{1}}X_{2}^{k_{2}}}{(a_{1}^{2}X_{1} +a_{2}^{2}X_{2})^{k}}.$$
	If $\tilde{D} \leq 1$, then $J_{2}(x)\leq 1$. 
	Thus we prove $\tilde{D} \leq 1$. 
	Since $\tilde{D}(X_{1},X_{2})=\tilde{D}(tX_{1},tX_{2})\ (t>0)$, in order to prove $\tilde{D}\leq 1$, 
	we show $\tilde{D}|_{P}\leq 1$ where 
	$$P=\{(X_{1},X_{2}) \in \mathbb{R}^{2} \mid X_{1}, X_{2} >0 ,\ a_{1}^{2}X_{1}+a_{2}^{2}X_{2}=1\}.$$
	We have $\tilde{D}|_{P}=X_{1}^{k_{1}}X_{2}^{k_{2}}$ and $X_{2}=\frac{1-a_{1}^{2}X_{1}}{a_{2}^{2}}$. 
	Since 
	\begin{eqnarray*}
		\frac{d \tilde{D}|_{P}}{d X_{1}}&=&k_{1}X_{1}^{k_{1}-1}X_{2}^{k_{2}}+X_{1}^{k_{1}}(-k_{2}\frac{a_{1}^{2}}{a_{2}^{2}})X_{2}^{k_{2}-1} =k_{1}X_{1}^{k_{1}-1}X_{2}^{k_{2}-1}(X_{2}-X_{1}),
	\end{eqnarray*}
	a critical point of $\tilde{D}|_{P}$ is only $X_{1}=1$ in $P$.
	Further, we get 
	$$\tilde{D}|_{P} \to 0 \ \ \text{as} \ X_{1}\to 0 \ \text{or} \ \frac{1}{a_{1}^{2}}.$$
	Hence $\max \{\tilde{D}(X_{1},X_{2})\mid (X_{1},X_{2}) \in P\} =\tilde{D}(1,1)=1$.
	Therefore 
	$$J_{2}(x) \leq 1. $$
	Then we have 
	\begin{eqnarray*}
	J(d\Phi)_{x}&=& J_{1}(x)J_{2}(x) =\|({\rm grad}f)_{x}\| J_{2}(x)\\
	&=&\frac{ \sqrt{a_{1}^{6}f_{2}(x_{2})^{4} J_{11}(x_{1})^2 +a_{2}^{6}f_{1}(x_{1})^{4} J_{12}(x_{2})^2 }}{(a_{2}^{3} f_{1}(x_{1}) +a_{1}^{3}f_{2}(x_{2}))^{2}} J_{21}(x_{1})J_{22}(x_{2})
	\frac{(a_{1}f_{2}(x_{2}))^{k_{1}} (a_{2}f_{1}(x_{1}))^{k_{2}}}{(a_{2}^{3}f_{1}(x_{1}) +a_{1}^{3} f_{2}(x_{2}))^{k}}\\
	&=&\frac{\sqrt{a_{1}^{6}f_{2}(x_{2})^{4} J_{11}(x_{1})^2J_{21}(x_{1})^{2}J_{22}(x_{2})^{2} +a_{2}^{6}f_{1}(x_{1})^{4} J_{12}(x_{2})^2 J_{21}(x_{1})^{2}J_{22}(x_{2})^{2}}}{(a_{2}^{3}f_{1}(x_{1}) +a_{1}^{3} f_{2}(x_{2}))^{k+2}}\\
	&&\times (a_{1}f_{2}(x_{2}))^{k_{1}} (a_{2}f_{1}(x_{1}))^{k_{2}}\\
	&\leq& \frac{\sqrt{a_{1}^{6}f_{2}(x_{2})^{4}  +a_{2}^{6}f_{1}(x_{1})^{4} }(a_{1}f_{2}(x_{2}))^{k_{1}} (a_{2}f_{1}(x_{1}))^{k_{2}}}{(a_{2}^{3}f_{1}(x_{1}) +a_{1}^{3} f_{2}(x_{2}))^{k+2}}\\
	&=&\frac{\sqrt{a_{1}^{2}X_{1}^{4}+a_{2}^{2}X_{2}^{4}}X_{1}^{k_{1}}X_{2}^{k_{2}}}{(a_{1}^{2}X_{1}+a_{2}^{2}X_{2})^{k+2}}.
	\end{eqnarray*} 
	We define 
$$D(X_{1},X_{2})=J(d\Phi_{x})^{2}=\frac{(a_{1}^{2}X_{1}^{4}+a_{2}^{2}X_{2}^{4})X_{1}^{2k_{1}}X_{2}^{2k_{2}}}{(a_{1}^{2}X_{1}+a_{2}^{2}X_{2})^{2k+4}}.$$
	We have $D(tX_{1},tX_{2})=D(X_{1},X_{2})\ (t>0)$.
	Similar to the above argument, we consider the maximum value of $D|_{P}$.
	Since 
	$$D|_{P}= (a_{1}^{2}X_{1}^{4}+a_{2}^{2}X_{2}^{4})X_{1}^{2k_{1}}X_{2}^{2k_{2}},$$
	we get 
	\begin{eqnarray*}
		\frac{d D|_{P}}{d X_{1}}&=& 4\left( a_{1}^{2}X_{1}^{3}-\frac{a_{1}^{2}}{a_{2}^{2}}a_{2}^{2}X_{2}^{3}\right) X_{1}^{2k_{1}}X_{2}^{2k_{2}}\\
		&&+(a_{1}^{2}X_{1}^{4}+a_{2}^{2}X_{2}^{4})\left( 2k_{1}X_{1}^{2k_{1}-1}X_{2}^{2k_{2}} -2k_{2}\frac{a_{1}^{2}}{a_{2}^{2}}X_{1}^{2k_{1}}X_{2}^{2k_{2}-1}\right)\\
		&=&-2a_{1}^{2}X_{1}^{2k_{1}-1}X_{2}^{2k_{2}-1}(X_{1}-X_{2}) \\
		&&\times \left\{ \left( (k_{1}-3)X_{1}^{4}+ (k_{2}-3)X_{2}^{4}\right)  +3(X_{1}-X_{2})^4+10(X_{1}-X_{2})^2 \right\}.
	\end{eqnarray*}
	Hence, if $k_{1}\geq 3,\ k_{2}\geq 3,$ then a critical point of $D|_{P}$ is only $X_{1}=1$ in $P$.
	Furthermore, we get
	$$D|_{P} \to 0 \ \ \text{as} \ X_{1}\to 0 \ \text{or} \ \frac{1}{a_{1}^{2}}.$$
	Thus $\max \{D(X_{1},X_{2})\mid (X_{1},X_{2}) \in P\} =D(1,1)=1$.
	Hence $D\leq 1$. This implies $J(d\Phi_{x}) \leq 1$.
	Therefore if $k_{1}\geq 3, k_{2}\geq 3,$ $\Phi$ is area nonincreasing.
	\end{proof}
	
\begin{rem}\rm
In 1969, Bombieri, DeGiorgi and Giusti \cite{BDG} showed that the cone over $S^{k}\times S^{k} \subset S^{2k+1} (k\geq 3)$
is area-minimizing.
On the other hand, Lawlor \cite{L} proved that the cone over $S^{k_{1}}\times S^{k_{2}} \subset S^{k_{1}+k_{2}+1}$
are not area-minimizing when $k_{1}+k_{2}\leq 5$ or $k_{1}=1, k_{2}=5$.
Hence, we need the condition $k_{1}\geq 3, k_{2}\geq 3$ in Theorem \ref{reduc}.
\end{rem}
	
\begin{rem}\rm
Area-nonincreasing retractions which we constructed in Section \ref{eg} satisfy the assumption of Theorem \ref{reduc}.
Moreover, an area-nonincreasing retraction that is constructed using Theorem \ref{reduc} satisfies the assumption of Theorem \ref{reduc} again.
Therefore, we can apply Theorem \ref{reduc} inductively. 
This implies that the cone over a product of two or more R-spaces with ``$\bigcirc$'' in the table in Section \ref{eg} is area-minimizing.
\end{rem}

\end{document}